\RequirePackage{fix-cm}
\documentclass[smallextended]{svjour3}       % onecolumn (second format)
\smartqed  % flush right qed marks, e.g. at end of proof

\usepackage{graphicx}
\usepackage{amssymb}
\usepackage[english]{babel}
\usepackage{mathtools}
\usepackage{microtype} 
\usepackage{verbatim}
\usepackage{mathrsfs}

\usepackage{todonotes}
\usepackage{tikz,pgfplots}
\usetikzlibrary{patterns}

\usepackage{comment}
\usepackage{subcaption}
\usepackage{environ}
\makeatletter
\providecommand{\env@tikzpicture@save@env}{}
\providecommand{\env@tikzpicture@process}{}

\usepackage{color}

\newcommand\xqed[1]{%
  \leavevmode\unskip\penalty9999 \hbox{}\nobreak\hfill
  \quad\hbox{#1}}
\newcommand{\exampleEnd}{\xqed{$\triangle$}}

\newcommand{\dR}{\mathbb{R}}

\newcommand{\calC}{\mathcal{C}}
\newcommand{\calF}{\mathcal{F}}

\newcommand{\calH}{\mathcal{H}}
\newcommand{\calI}{\mathcal{I}}
\newcommand{\calM}{\mathcal{M}}

\newcommand{\calV}{\mathcal{V}}
\newcommand{\calW}{\mathcal{W}}

\newcommand{\scrC}{\mathscr{C}}

\newcommand{\bx}{\mathbf{x}}
\newcommand{\by}{\mathbf{y}}

\renewcommand{\phi}{\varphi}
\renewcommand{\epsilon}{\varepsilon}
\providecommand{\abs}[1]{\lvert#1\rvert}
\providecommand{\norm}[1]{\lVert#1\rVert}

\DeclareMathOperator{\Conv}{conv}

 \newcommand{\R}{{\mathbb R}}

 \newcommand{\Ccal}{\mathcal{C}}
 
 \newcommand{\Fcal}{\mathcal{F}}

 \newcommand{\Wcal}{\mathcal{W}}
 \newcommand{\Xcal}{\mathcal{X}}

 \newcommand{\xb}{\bar{x}}
 \newcommand{\yb}{\bar{y}}

 \newcommand{\bi}{\begin{itemize}}
 	\newcommand{\ei}{\end{itemize}}
 \newcommand{\bpm}{\begin{pmatrix}}
 	\newcommand{\epm}{\end{pmatrix}}

 \newcommand{\nrm}[1]{|#1|}

 \newcommand{\ssoucin}[2]{\langle #1,#2\rangle}

 \newcommand{\conv}{\operatorname{conv}}

\newcommand{\Ccalv}{{\mathcal{C}(v)}}
\newcommand{\Mcalv}{{\mathcal{M}(v)}}
\newcommand{\Wcalv}{{\mathcal{W}(v)}}

\newcommand{\bxb}{{\bar{\bx}}}
\newcommand{\Gl}{\hat{v}}

\newcommand{\sdf}{{\hat{\partial}}}
\newcommand{\sdl}{{\partial}}
\newcommand{\sdc}{{\overline{\partial}}}

\newcommand{\la}{\lambda}
\newcommand{\one}{\chi_N}
\newcommand{\zero}{0}

\begin{document}
\title{The Intermediate Set and Limiting Superdifferential for Coalitional Games: Between the Core and the Weber Set\thanks{We wish to express our gratitude to the reviewers and the editor for the challenging remarks concerning our paper. L. Adam gratefully acknowledges the support from the Grant Agency of the Czech Republic (15-00735S). The work of T. Kroupa was supported by Marie Curie Intra-European Fellowship OASIG (PIEF-GA-2013-622645).}
}

\titlerunning{The Intermediate Set and Limiting Superdifferential for Coalitional Games}        % if too long for running head

\author{Luk\'{a}\v{s} Adam  \and Tom\'{a}\v{s} Kroupa}


\institute{L. Adam \at
	Institute of Information Theory and Automation of the Czech Academy of Sciences,\\ Pod Vod\'{a}renskou v\v{e}\v{z}\'{i} 4, 182 08 Prague, Czech Republic\\
	\email{adam@utia.cas.cz}           %  \\
	\and
	T. Kroupa \at
	Dipartimento di Matematica ``Federigo Enriques'', Universit\`a degli Studi di Milano,\\ Via Cesare Saldini 50, 20133 Milano, Italy\\
	\email{Tomas.Kroupa@unimi.it}
}

\date{Received: date / Accepted: date}

\maketitle

\begin{abstract}
We introduce the intermediate set as an interpolating solution concept between the core and the Weber set of a coalitional game. The new solution is defined as the limiting superdifferential of the Lov\' asz extension and thus it
completes the hierarchy of variational objects used to represent the core (Fr\'echet superdifferential) and the Weber set (Clarke superdifferential). It is shown that the intermediate set is a~non-convex solution containing the Pareto optimal payoff vectors that depend on some chain of coalitions and marginal coalitional contributions with respect to the chain. A detailed comparison between the intermediate set and other set-valued solutions is provided. We compute the exact form of intermediate set for all games and provide its simplified characterization for the simple games and the glove game.
\end{abstract}

\section{Introduction}
Many important solution concepts for transferable-utility $n$-person coalitional games can be equivalently expressed as formulas involving gradients or generalized gradients of a~suitable  extension of the given game. This applies to some of the well-known single-valued solutions, such as the Shapley value and the Banzhaf--Coleman index of power. These constructions usually rely on the multilinear extension of coalitional games as functions from the discrete cube $\{0,1\}^n$ onto $[0,1]^n$; see \cite[Chapter XII]{Owen95}, for example. The purpose of such a ``differential representation'' of the solution is not only computational, but it is also to provide a~new interpretation of the corresponding payoff vectors, which usually revolves around the idea of marginal contributions to a~given coalition. 

The recent progress in variational analysis \cite{mordukhovich.2006,rockafellar.wets.1998} enables us to construct various kinds of generalized derivatives, the so-called subgradients and supergradients, for a very large family of lower semicontinuous functions. The concept of a gradient of a~differentiable function is replaced by that of a subdifferential (superdifferential) of a possibly nonsmooth function. The elements of a superdifferential---the supergradients---have a close geometric connection with Jacobians of all smooth majorants of the function at the neighborhood of a given point; see Appendix \ref{App:Super}. The Fr\'echet superdifferential, the limiting (Mordukhovich) superdifferential and the Clarke superdifferential count among the main superdifferentials used in nonsmooth analysis.

The representation of some solution concepts by generalized derivatives for selected classes of cooperative games was studied already by Aubin \cite{Aubin74a}. The authors of \cite{danilov.koshevoy.2000,Sagara14} use the Lov\'asz extension of a coalitional game in order to express the core and the Weber set in terms of its Fr\'echet and the Clarke superdifferential, respectively. 

In this paper we pursue a converse research direction by adopting the idea proposed in \cite{Sagara14}: We employ the limiting superdifferential to define a new solution concept for coalitional games, the so-called intermediate set. Specifically, the intermediate set is the limiting superdifferential of the Lov\'asz extension for the grand coalition. The associated payoff vectors can be thus interpreted as marginal contributions to the grand coalition. However, several questions arise at this point, for instance:
\begin{itemize}
	\item  What is an interpretation and properties of the intermediate set?
	\item  Are the payoff vectors in the intermediate set determined by some reasonable principles of profit allocation?
\end{itemize}

The main goal of this paper is to argue that the newly constructed solution is sensible and interesting from many perspectives. Using the tools of variational analysis, we will show that the intermediate set
\begin{itemize}
	\item is a nonempty, subadditive and Pareto optimal solution,
	\item is a finite union of (possibly empty) convex polytopes, each of which can be interpreted as a~core of some ``marginal'' game,
	\item lies in-between the core and the Weber set,
	\item coincides with the core if and only if the game is supermodular.
\end{itemize}

The intermediate set can be viewed as a nonempty interpolant between the core and the Weber set, which is convenient especially whenever the former is very small and the latter is huge. Our Theorem \ref{Theorem Bv} provides a clear interpretation of the payoff vectors from the intermediate set: for some coalitional chain, each such vector is a Weber-style marginal vector on the level of blocks of coalitions and, at the same time, no coalition inside each block can improve upon this payoff vector in the sense of marginal coalitional contributions. Hence, the intermediate set is a solution concept that looks globally like the Weber set, but behaves locally like the core concept. In Theorem \ref{Thm IS Cores} we will show that the intermediate set is a non-convex polyhedron whose convex components coincide with cores of some games, which are determined by a given coalitional chain and marginal coalitional contributions. As for examples, the intermediate set on the class of simple games and glove games is computed in Theorem \ref{Theorem simple Bv} and Theorem \ref{Theorem gloves}, respectively.

The article is structured as follows. We fix our notation and terminology in Section~\ref{sec:coreweb}, where we repeat basic facts about solution concepts, the Lov\' asz extension and its superdifferentials. Section \ref{sec:IS} contains a characterization of the intermediate set based on coalitional chains in the player set (Theorem \ref{Theorem Bv}) and discussion of a~distribution process that leads to a payoff vector in the intermediate set. Some motivating examples are also included (Examples \ref{Example gloves simple} and~\ref{Example star}). We carry out an in-depth inspection of the properties of the intermediate set and compare it to the various solution concepts in Section \ref{Section properties}. Differences among the core, the intermediate set and the Weber set are summarized in Table~\ref{Table properties}. Two selected classes of coalitional games (the simple games and the glove game) are analyzed in Section \ref{sec:ex}  in order to refine a formula from Theorem~\ref{Theorem Bv}. The main part of the paper is concluded with an outlook towards further research in Section \ref{sec:concl}. Appendix consists of two parts. A brief explanation of the notions from nonsmooth analysis is in Appendix~\ref{App:Super}, with no attempt at a~comprehensive discussion of all the results needed in the paper. Appendix \ref{Appendix proof} contains the proof of our main result, Theorem~\ref{Theorem Bv}.

\section{Core and Weber Set}\label{sec:coreweb}

We use the standard notions and results from cooperative game theory; see \cite{PelegSudholter07}.
Let $N=\{1,\dots,n\}$ be a finite set of \emph{players}, where $n$ is a positive integer. A~\emph{coalition} is a subset $A\subseteq N$ and by $2^N$ we denote the powerset of $N$.
A \emph{coalitional game (with transferable utility)} is a function $v\colon 2^N\to \dR$ with $v(\emptyset)=0$. Any  $\bx=(x_1,\dots,x_n)\in \dR^n$ is called a \emph{payoff vector}. We introduce the following notation:
\[
\bx(A)=\sum_{i\in A} x_i, \quad \text{for every $A\subseteq N$.}
\]
We say that a payoff vector $\bx$ is \emph{feasible} in a game $v$ whenever $\bx(N)\leq v(N)$. The set of all feasible payoff vectors in $v$ is denoted by $\calF(v)$.

Let $\Gamma(N)$ be the set of all games and $\Omega\subseteq \Gamma(N)$.  A \emph{solution} on $\Omega$ is a set-valued mapping $\sigma\colon \Omega\to 2^{\dR^n}$ that sends every game $v\in\Omega$ to a set $\sigma(v)\subseteq\Fcal(v)$. We recall the core solution and the Weber set. The \emph{core} of a~game $v$ is the convex polytope 
\[
\calC(v)=\{\bx \in \dR^n \mid \bx(N)=v(N),\; \bx(A)\geq v(A) \text{ for every $A\subseteq N$}\}.
\]
The core is always contained in the set $\calI(v)$ of all \emph{imputations} in the game $v$,
\begin{equation}
\calI(v)=\{\bx\in \R^n\mid \bx(N)=v(N),\, x_i\geq v(\{i\}),\, i\in  N\}. \label{imputations}
\end{equation}

Let $\Pi_n$ be the set of all the permutations $\pi$ of the player set $N$. Let $v\in \Gamma(N)$ and $\pi \in \Pi_n$. A~\emph{marginal vector} of a game $v$ with respect to  $\pi$ is the payoff  vector $\bx^v(\pi)\in\dR^n$ with coordinates
\begin{equation}\label{def:MargVec}
x^v_i(\pi)=v\left(\bigcup_{j\leq \pi^{-1}(i)}\{\pi(j)\}\right)-v\left(\bigcup_{j< \pi^{-1}(i)}\{\pi(j)\}\right), \quad i\in N.
\end{equation}
The \emph{Weber set} of $v$ is the convex hull of all the marginal vectors of $v$,
\[
\calW(v)=\Conv \{\bx^v(\pi)\mid \pi \in \Pi_n\}.
\]
Since $\bx^v(\pi)(N)=v(N)$, the Weber set is a solution on $\Gamma(N)$. Moreover, the inclusion $\calC(v)\subseteq \calW(v)$ holds true for every $v\in \Gamma(N)$; see \cite[Theorem 14]{Weber88}.

The fundamental tool in this paper is the concept of Lov\'{a}sz extension~\cite{Lovasz83}. For every set $A\subseteq N$ let $\chi_A$ denote the incidence vector in $\dR^n$ whose coordinates are given by
\begin{equation}\label{defincidence}
(\chi_A)_i=
\begin{cases}
1 & \text{if $i\in A$,}\\ 0 & \text{otherwise.}
\end{cases}
\end{equation}
We write $0$ in place of $\chi_{\emptyset}$. The embedding of $2^N$ into $\dR^n$ by means of the mapping $A\mapsto \chi_A$ makes it possible to interpret a game on  $2^N$ as a real function on $\{0,1\}^n$. 
Indeed, it is enough to define $\hat{v}(\chi_A)=v(A)$ for every $A\subseteq N$. The function $\hat{v}$ can be extended onto the whole of $\dR^n$ as follows. For every $\bx\in\dR^n$, put
\begin{equation*}\label{Label defin Pix}
\Pi(\bx)=\{\pi\in \Pi_n\mid x_{\pi(1)}\geq \dots \geq x_{\pi(n)}\}. 
\end{equation*}
Given $i\in N$ and $\pi\in \Pi(\bx)$, define
\begin{equation*}%\label{Label defin Vipi}
V_i^{\pi}(\bx)=\{j\in N\mid x_j\geq x_{\pi(i)}\}.
\end{equation*}
Note that $V_i^{\pi}(\bx)=V_i^{\rho}(\bx)$ for every $\pi,\rho\in\Pi(\bx)$. This implies that any vector $\bx\in\dR^n$ can be unambiguously written as a linear combination
\begin{equation}\label{eq:decomposition}
\bx=\sum_{i=1}^{n-1}(x_{\pi(i)}-x_{\pi(i+1)})\cdot \chi_{V_i^{\pi}(\bx)}+x_{\pi(n)}\cdot \chi_N
\end{equation}
for an arbitrary $\pi\in \Pi(\bx)$. It is convenient to define $V_0^{\pi}(\bx):=\emptyset$. Then we can rewrite (\ref{eq:decomposition}) as
\begin{equation}\label{eq:decomposition2}
\bx=\sum_{i=1}^{n}x_{\pi(i)}\cdot \left (\chi_{V_i^{\pi}(\bx)}-\chi_{V_{i-1}^{\pi}(\bx)}\right). 
\end{equation}
The \emph{Lov\'asz extension} $\hat{v}$ of $v\in\Gamma(N)$ is the function $\dR^n\to\dR$ defined linearly with respect to the decomposition (\ref{eq:decomposition2}):%\todo{pridat ze to nezavisi na volbe $\pi\in\Pi(\bx)$}
\begin{equation}\label{Label defin Lovasz}
\hat{v}(\bx)=\sum_{i=1}^{n}x_{\pi(i)}\cdot \left (v(V_i^{\pi}(\bx))-v(V_{i-1}^{\pi}(\bx))\right), \quad \text{for any $\bx\in\dR^n$.}
\end{equation}
Observe that the definition of $\hat{v}(\bx)$ is independent on the choice of $\pi\in\Pi(\bx)$.
Clearly $\hat{v}(\chi_A)=v(A)$ for every coalition $A\subseteq N$.
% Let $\calO$ be a chain\footnote{By a \emph{chain} in $2^N$ we mean a totally ordered subset of $2^N$.} in $2^N$ containing $N$. We may list the elements of $\calO$ as $A_1,\dots,A_k$ with $A_1=N$, where $k=\abs{\calO}$. With each such chain $\calO$ we associate the cone\footnote{A \emph{cone} is a nonempty subset of a real linear space closed under nonnegative scalar multiplication.} defined as follows:
%\[
%\Cone(\calO)=\left\{\sum_{i=1}^k \alpha_i\cdot\chi_{A_i} \mid \alpha_1\in\dR,\; \alpha_2\dots,\alpha_k\geq 0\right\}.
%\]
It is easy to see that the Lov\'asz extension $\hat{v}$ of any game $v$ fulfills these properties:
\begin{itemize}
\item $\hat{v}$ is continuous and piecewise affine on $\dR^n$;
\item $\hat{v}$ is positively homogeneous: $\hat{v}(\lambda\cdot \bx)=\lambda\cdot \hat{v}(\bx)$  for every $\lambda\geq 0$ and $\bx\in\dR^n$;
\item the mapping $v\in \Gamma(N) \mapsto \hat{v}$ is linear.
\end{itemize}
\noindent
The following easy lemma says that the local behavior of $\hat{v}$ is the same around $\chi_N$ as in the neighborhood of $0$.

\begin{lemma}\label{Lemma subdifferential zero one}
	For any $\bx\in\R^n$ it holds true that
	$$
	\Gl(\bx+\one) = \Gl(\bx)+\Gl(\one).
	$$
\end{lemma}
\begin{proof}
This follows directly from the definition (\ref{Label defin Lovasz}) together with the identities $\Pi(\bx+\one)=\Pi(\bx)$, $\Pi(\chi_N)=\Pi_n$, and $V_1^{\pi}(\chi_N)=\ldots=V_n^{\pi}(\chi_N)=N$ for every $\pi\in \Pi_n$.
\qed\end{proof}

A game $v\in\Gamma(N)$ is called \emph{supermodular} (or \emph{convex}) if the following inequality is satisfied:
\[
v(A\cup B)+v(A\cap B)\geq v(A)+v(B),\quad \text{for every $A,B\subseteq N$.}
\]
Every supermodular game $v$ is also \emph{superadditive}:
\[
v(A\cup B)\geq v(A)+v(B),\quad \text{for every $A,B\subseteq N$ with $A\cap B=\emptyset$.}
\]
A game $v$ is \emph{submodular}  if the game $-v$ is supermodular. A game $v$ is called \emph{additive} when $v(A\cup B)=v(A)+v(B)$ for every $A,B\subseteq N$ with $A\cap B=\emptyset$.  We will make an ample use of several characterizations of supermodular games appearing in the literature. 

\begin{proposition}\label{Proposition supermodularity equivalence}
Let $v\in\Gamma(N)$. Then the following are equivalent:
\begin{enumerate}
\item $v$ is supermodular;
\item $\{\bx^v(\pi)\mid \pi \in \Pi_n\}\subseteq \calC(v)$;
\item $\calC(v)=\calW(v)$;
\item The Lov\' {a}sz extension $\hat{v}$ of $v$ is a concave function.
\end{enumerate}
\end{proposition}
\begin{proof}
Shapley \cite{Shapley72} proved 1. $\Rightarrow$ 2. and Weber \cite{Weber88} showed that 2. $\Rightarrow$ 3., respectively.
The implication 3. $\Rightarrow$ 1. was shown by Ichiishi \cite{Ichiishi81}.
The equivalence between 1. and 4. is the ``supermodular" version of the theorem originally proved by Lov\'asz in \cite{Lovasz83} for submodular games.
\qed\end{proof}
\begin{remark}
An extensive survey of the conditions equivalent to supermodularity together with (references to) the proofs can be found in \cite[Appendix A]{StudenyKroupa:CoreExtreme}. The notion of ``convexity" is somewhat overloaded in the game-theoretic literature since it appears in a number of different contexts and meanings. On top of that convex games have concave Lov\'asz extensions. For those reasons we strictly prefer the term ``supermodular game" over ``convex game", although the latter is commonly used.
\exampleEnd\end{remark}

The Lov\'asz extension $\Gl$ of a coalitional game $v$ is used to characterize the core and the Weber set by the tools of nonsmooth calculus. It was shown in \cite[Proposition~3]{danilov.koshevoy.2000} that the core coincides with the Fr\'echet superdifferential of $\Gl$ at~$\zero$, $\Ccalv=\sdf \Gl(\zero)$. Similarly, from \cite[Proposition~4.1]{Sagara14} we know that the Weber set is the Clarke superdifferential of $\Gl$ at~$\zero$,  $\Wcalv = \sdc\Gl(\zero)$. From the viewpoint of game theory, however, it is more sensible to evaluate the superdifferentials of~$\Gl$ at $\one$ since it conforms with the idea of marginal contributions to the grand coalition $N$. This is possible by Lemma \ref{Lemma subdifferential zero one}  so that we can shift the computations of the respective superdifferentials to~$\chi_N$.

\begin{proposition}\label{Proposition Cv Wv superdifferential}
For every game $v\in\Gamma(N)$,
$$
\aligned
\Ccalv &= \sdf \Gl(\one) = \sdf \Gl(\zero),\\
\Wcalv &= \sdc \Gl(\one) = \sdc \Gl(\zero).
\endaligned
$$
\end{proposition}

\section{Intermediate Set}\label{sec:IS}

This section is composed of three subsections. In the first one we define the intermediate set using the limiting superdifferential. The characterization based on coalitional chains is proved in Subsection~\ref{subsec:OP}. In the last subsection we show that the intermediate set can be expressed as a union of cores of certain marginal games.

\subsection{Definition and basic properties}

As we already noted in the introduction, it can frequently happen that the core is small or empty and, at the same time, the Weber set is too coarse. For this reason we follow the idea of Boris Mordukhovich, which was mentioned in \cite{Sagara14}, and we define a new solution concept directly as $\sdl \Gl(\one)$ by analogy with Proposition \ref{Proposition Cv Wv superdifferential}, where $\sdl$ is the limiting superdifferential. The limiting superdifferential always lies in-between the Clarke superdifferential and the Fr\'echet superdifferential (see Appendix~\ref{App:Super}). A straightforward interpretation of the limiting superdifferential is that it coincides with the union of all Fr\'echet superdifferentials with respect to some sufficiently small neighborhood of the point in question. 

\begin{definition}\label{def:IS}
	Let $v\in\Gamma(N)$. The \emph{intermediate set} $\Mcalv$ of $v$ is the set
$$
\Mcalv := \sdl\Gl(\one).
$$
\end{definition}
In Subsection~\ref{subsec:OP} we will derive a combinatorial formula for $\Mcalv$, which bypasses the computation of Lov\' asz extension and the limiting superdifferential. The following example shows the shape of $\Mcalv$ for a particular non-supermodular $3$-player game $v$.

\begin{example}\label{Example gloves simple}
Consider a game with the player set $N=\{1,2,3\}$ in which the first player owns a single left glove, while the remaining two players possess one right glove each. The profit of a coalition $A\subseteq N$ is the number of glove pairs the coalition owns:
$$
v(A) = \begin{cases} 1 &\text{if }A\in\{\{1,2\}, \{1,3\}, N\}, \\ 0 &\text{otherwise.} \end{cases}
$$
It is not difficult to compute $\Ccalv, \Mcalv$ and $\Wcalv$ directly by the definition. However, since $v$ is both a~simple game and a~glove game, we can also employ Theorem \ref{Theorem simple Bv} and Theorem \ref{Theorem gloves} to recover~$\Mcalv$. Thus, 
$$
\aligned
\Ccalv &= \{(1,0,0)\},\\
\Mcalv &= \conv\{(1,0,0),(0,1,0)\} \cup \conv \{(1,0,0),(0,0,1)\},\\
\Wcalv &= \conv\{(1,0,0),(0,1,0),(0,0,1)\}.\\
\endaligned
$$
\begin{figure}[!ht]
\hspace{.03\textwidth}
\begin{subfigure}[b]{.45\textwidth}
\centering
\begin{tikzpicture}[scale=1.5]
\draw [thick] (0,0) -- (2,0) node[below] {$x_2$};
\draw [thick] (0,0) -- (0,2)  node[right] {$x_1$};
\draw [thick] (0,0) -- (-1,-1) node[left] {$x_3$};
\draw [very thick] (1,0) -- (0,1) -- (-0.5,-0.5);
\filldraw (1,0) circle (0.6mm);
\filldraw (0,1) circle (0.6mm);
\filldraw (-0.5,-0.5) circle (0.6mm);
\end{tikzpicture}
\caption*{Intermediate set}
\label{Figure glove 1}
\end{subfigure}
\hspace{.03\textwidth}
\begin{subfigure}[b]{.45\textwidth}
\centering
\begin{tikzpicture}[scale=1.5]
\draw [thick] (0,0) -- (2,0) node[below] {$x_2$};
\draw [thick] (0,0) -- (0,2)  node[right] {$x_1$};
\draw [thick] (0,0) -- (-1,-1) node[left] {$x_3$};
\draw [pattern = north west lines] (1,0) -- (0,1) -- (-0.5,-0.5) -- (1,0);
\filldraw (1,0) circle (0.6mm);
\filldraw (0,1) circle (0.6mm);
\filldraw (-0.5,-0.5) circle (0.6mm);
\end{tikzpicture}
\caption*{Weber set}
\label{Figure glove 2}
\end{subfigure}
\hspace{.03\textwidth}
\caption{The intermediate set and the Weber set for the $3$-person glove game}
\label{Figure glove}
\end{figure}
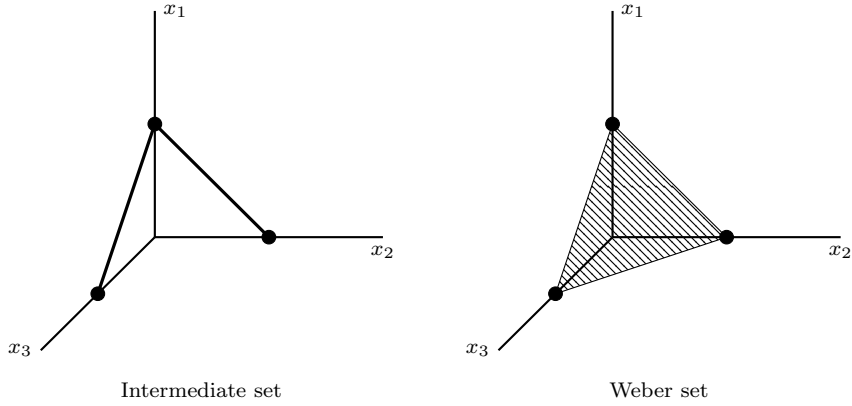
We will briefly comment on the shape of the solutions. The only payoff vector in the core $\Ccalv$ assigns the total worth to player 1. This allocation reflects the principle of stability: the surplus of right-hand gloves on the market makes both player 2 and 3 accept arbitrarily small payoff. On the other hand, the Weber set $\Wcalv$ coincides with the set of all imputations $\calI(v)$, which may be difficult to interpret. The intermediate set $\Mcalv$ allows for two scenaria, each of which involves two players only: player $1$ strikes a deal either with player $2$ or with player $3$. Once such a two-player coalition $\{1,i\}$ arises, where $i\in \{2,3\}$, the coalition has the effective power to distribute its profit to $1$ and $i$ in an arbitrary ratio. The remaining player (a~non-contractor) is therefore eliminated from any allocation process.
\exampleEnd\end{example}

In the rest of this section, we will show  basic properties of $\Mcalv$.

\begin{lemma}\label{Lemma Icalv basic properties}
Let $v\in\Gamma(N)$. Then:
\begin{enumerate}
	\item $\Mcalv\neq \emptyset$.
	\item  We have  
	\begin{equation}\label{Label sets inclusions}
	\Ccalv\subseteq\Mcalv\subseteq \Wcalv,
	\end{equation}
	where both inclusions may be strict.
	\item $\Wcalv = \conv\Mcalv$.
	\item $v$ is supermodular if and only if $\Ccalv=\Mcalv$. 
\end{enumerate}
\end{lemma}
\begin{proof}
By \cite[Corollary 8.10, Theorem 9.13]{rockafellar.wets.1998} we have $\Mcalv\neq\emptyset$. The inclusions~\eqref{Label sets inclusions} follow from the relation  $\sdf f(x)\subseteq \sdl f(x) \subseteq \sdc f(x)$% and the equalities from Proposition \ref{Proposition Cv Wv superdifferential}.
; see Appendix \ref{App:Super}. Item 3. is a consequence of \cite[Theorem~8.49]{rockafellar.wets.1998}.  Supermodularity of $v$ is equivalent to $\Ccalv=\Wcalv$ by Shapley--Ichiishi theorem \cite{Shapley72,Ichiishi81}. Thus, item~4. follows from 3. together with convexity of $\Ccalv$.
\qed\end{proof}
\noindent
Analogously to Proposition \ref{Proposition Cv Wv superdifferential} we can evaluate the limiting superdifferential at~$0$ and still obtain the same result, $\Mcalv$.
\begin{lemma}\label{Lemma sets zero one}
The following identity is satisfied for every game $v\in \Gamma(N)$:
$$
\Mcalv = \sdl \Gl(\one) = \sdl \Gl(\zero).
$$
\end{lemma}
\begin{proof}
	By Lemma \ref{Lemma subdifferential zero one} the Lov\'asz extension $\Gl$ has the same structure in the neighborhood of $\one$ and in the neighborhood of $\zero$. \qed
\end{proof}
\noindent
Putting together Proposition \ref{Proposition Cv Wv superdifferential} and Lemma \ref{Lemma sets zero one}, we can now summarize the relations between the discussed solutions and the superdifferentials as follows:
\begin{equation*}
\aligned
\Ccalv &= \sdf \Gl(\one) = \sdf \Gl(\zero),\\
\Mcalv &= \sdl \Gl(\one) = \sdl \Gl(\zero),\\
\Wcalv &= \sdc \Gl(\one) = \sdc \Gl(\zero).
\endaligned
\end{equation*}

\subsection{Characterization by chains}\label{subsec:OP}
In this section we are going to prove the main characterization of the intermediate set,  Theorem \ref{Theorem Bv}. Its purpose is twofold. First, this result shows that the purely analytic definition of intermediate set can be equivalently stated in terms of the combinatorial and order-theoretic properties of a~coalitional game. Second, it may be better to use Theorem \ref{Theorem Bv} than the definition based on the limiting superdifferential for the computational reasons. In what follows the main tool is the notion of a coalitional chain, which can be thought of as a~generalization of a permutation (or a total order) on the player set $N$.

A \emph{(coalitional) chain} is a subset $\calH=\{C_1,\dots,C_k\}$ of $2^N$, such that $k \geq 1$, $C_1\neq\emptyset$, $C_i\subsetneq C_{i+1}$ for $i=1,\dots,k-1$, and $C_k=N$. We will assume $C_0:=\emptyset$ throughout the paper. Let $\scrC$ be the set of all coalitional chains in $2^N$. The family $\scrC$ is associated with the following scheme of allocating payoffs $\bx\in\dR^n$ among the players in a game $v$:
\begin{enumerate}
	\item The players are organized into a chain $\calH=\{C_1,\dots,C_k\}$, depending on their position in the allocation process.
	\item Each coalition $C_i\setminus C_{i-1}$ can distribute the total amount
	\begin{equation}\label{Label Bv 1}
	\bx(C_i\setminus C_{i-1})=v(C_i) - v(C_{i-1})
	\end{equation}
	to its members, for all $i=1,\dots,k$. This can be interpreted as the marginal contribution of $C_i\setminus C_{i-1}$ to the coalition $C_{i-1}$ with respect to the chain $\calH$.
	\item No coalition $B\subseteq C_i\setminus C_{i-1}$ can improve upon $\bx$ while respecting the  order of coalitions given by $\calH$, that is,
	\begin{equation}
	\label{Label Bv 2}
	\bx(B) \ge v(C_{i-1}\cup B) - v(C_{i-1})
	\end{equation}
for all $B\subseteq C_i\setminus C_{i-1}$ and all $i=1,\dots,k$.
\end{enumerate}

Note that the players share the total of $v(N)$ among them as a consequence of item 2. Our main result says that $\bx\in \Mcalv$ if and only if there exists a chain~$\calH$ such that the total profit $v(N)$ is distributed among the players according to items 1.--3. above. For any chain ${\calH}=(C_1,\dots,C_k)$, put
\[
\calM_{\calH}(v) = \{\bx\in\R^n \mid \text{$\bx$ satisfies \eqref{Label Bv 1} and \eqref{Label Bv 2}}\}.
\]

\begin{theorem}\label{Theorem Bv}
Let $v\in\Gamma(N)$. Then
\begin{equation}\label{eqThm1}
\Mcalv=\bigcup_{\calH\in\scrC} \calM_{\calH}(v).
\end{equation}
\end{theorem}
\begin{proof}
See Appendix \ref{Appendix proof}. \qed
\end{proof}

\noindent
Theorem \ref{Theorem Bv} can serve as an alternative definition of $\Mcalv$. Since the union in~(\ref{eqThm1}) runs over $\scrC$, computing $\Mcalv$ can be a fairly complex task---it is known that the cardinality of $\scrC$ equals the $n$-th ordered Bell number. For example, $|\scrC|=75$ for $n=4$.

\begin{remark}
	In order to simplify the notation for coalitions we will occasionally omit the braces and commas, so that 
	a coalition $\{i,j\}$ is written as $ij$.
	\exampleEnd\end{remark}

The distribution procedure satisfying the conditions 1.--3. above has two extreme cases. Assume that the chain $\calH$ is maximal, that is, for some permutation $\pi\in\Pi_n$ we have $\calH=\{\pi(1),\pi(1)\pi(2),\dots,\pi(1)\dots\pi(n)\}$ . In this case the profit allocation in any game $v$ leads to a single marginal vector $\bx^v(\pi)$ defined by (\ref{def:MargVec}), $\calM_{\calH}(v)=\{\bx^v(\pi)\}$. On the contrary, if the chain is $\{N\}$, then all the players (and coalitions) are treated equally, which results in distributing payoffs according to the definition of core, $\calM_{\{N\}}(v)=\calC(v)$. Any chain $\calH=\{C_1,\dots,C_k\}$ different from those two borderline cases generates allocations $\bx\in \calM_{\calH}(v)$ combining the rule of marginal contributions for coalitions $C_i\setminus C_{i-1}$ (item 2.) with the core-like stability for sub-coalitions $B$ of $C_i\setminus C_{i-1}$ (item 3.).

We will now present two examples. First, we will make use of Theorem~\ref{Theorem Bv} to write the formula for the intermediate set of any $3$-player coalitional game. Second, using the first example we will present a game for which the three considered solution concepts differ substantially. 

\begin{example}\label{Example game n=3}
	Let $N=\{1,2,3\}$. The family $\scrC$ of all chains has $13$ elements in this case:
	\begin{equation*}
	\begin{split}
	\scrC =& \bigl\{\{N\},\{1,N\},\{2,N\},\{3,N\},\{12,N\},\{13,N\},\{23,N\}\bigr\}\;  \cup \\  & \bigl\{\{\pi(1),\pi(1)\pi(2),N\}\mid \pi \in \Pi_n\bigr\}.
		\end{split}
	\end{equation*}
	Let $v\in\Gamma(N)$. For example, the choice $\calH=\{1,N\}$ gives
	\begin{equation*}
	\begin{split}
	\calM_{\calH}(v)= \bigl\{\bx\in\R^3 \mid  & x_1=v(1),\, \bx(23)= v(N	)-v(1),\\ & x_2\geq v(12)-v(1),\, x_3\geq v(13)-v(1)\bigr\}.
		\end{split}
		\end{equation*}
	Theorem \ref{Theorem Bv} says that 
	\begin{align*}
	\Mcalv=\phantom{ }&\Ccalv\ \cup\\ &  \calM_{\{1,N\}}(v) \cup  \calM_{\{2,N\}}(v)\cup \calM_{\{3,N\}}(v)\ \cup   \\ &  \calM_{\{12,N\}}(v) \cup \calM_{\{13,N\}}(v)\cup \calM_{\{23,N\}}(v)\ \cup   \\ & \{\bx^v(\pi)\mid \pi \in \Pi_n\}.
	\end{align*}
\exampleEnd\end{example}

\begin{example}\label{Example star}
	Let $N=\{1,2,3\}$ and
	$$
	v(A)=\begin{cases} 0 &\text{if $\nrm{A}=1$}, \\ 2 &\text{if $\nrm{A}=2$}, \\ 3 &\text{if $A=N$}.\end{cases}
	$$
	It is easy to see that $v$ is superadditive but not supermodular.
	\begin{figure}[!ht]
	\begin{subfigure}[b]{.3\textwidth}
	\centering
	\begin{tikzpicture}[scale=0.84]
	\draw [gray] (-2,0) -- (2,0) -- (0,3.46) -- cycle;
	\filldraw (0,1.15) circle (0.6mm);
	\end{tikzpicture}
	\caption{Core}
	\label{Figure star 1}
	\end{subfigure}
	\hspace{.03\textwidth}
	\begin{subfigure}[b]{.3\textwidth}
	\centering
	\begin{tikzpicture}[scale=0.84]
	\draw [gray] (-2,0) -- (2,0) -- (0,3.46) -- cycle;
	\draw [ultra thick] (1.33,1.15) -- (-1.33,1.15);
	\draw [ultra thick] (-0.67,0) -- (0.67,2.31);
	\draw [ultra thick] (0.67,0) -- (-0.67,2.31);
	\end{tikzpicture}
	\caption{Intermediate set}
	\label{Figure star 2}
	\end{subfigure}
	\hspace{.03\textwidth}
	\begin{subfigure}[b]{.3\textwidth}
	\centering
	\begin{tikzpicture}[scale=0.84]
	\draw [gray] (-2,0) -- (2,0) -- (0,3.46) -- cycle;
	\filldraw (-0.67,0) -- (0.67,0) -- (1.33,1.15) -- (0.67,2.31) -- (-0.67,2.31) -- (-1.33,1.15) -- cycle;
	\end{tikzpicture}
	\caption{Weber set}
	\label{Figure star 3}
	\end{subfigure}
	\caption{The solutions from Example \ref{Example star} in the barycentric coordinates}
	\label{Figure star}
	\end{figure}
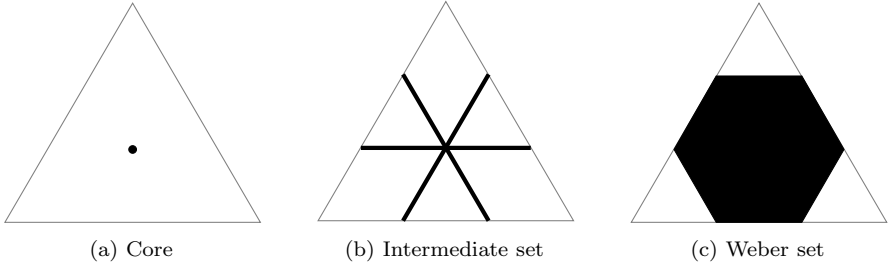
	\noindent
The core of this game is single-valued, $\Ccalv = \{(1,1,1)\}$, whereas the Weber set~$\calW(v)$ is the hexagon whose $6$ vertices are all the coordinate-wise permutations of the payoff vector $(0,1,2)$. The intermediate set is the union of three line segments---see Figure \ref{Figure star}(b). We obtain that $\calM_{\{i,N\}}(v)=\emptyset$ for every chain $\{i,N\}$. On the other hand, any component $\calM_{\{ij,ijk\}}(v)$ is the line segment whose endpoints are the two marginal vectors $\bx$ with $x_k=1$. Thus a payoff vector $\bx\in\dR^3$ is in $\Mcalv$ iff it belongs to $\calM_{\{ij,ijk\}}(v)$ for some chain $\{ij,ijk\}$. Note that this example also shows that, in general, the intermediate set is not a~union of some faces of the Weber set.
\exampleEnd\end{example}

Lemma \ref{Lemma Icalv basic properties} states that $v$ is supermodular if and only if the core coincides with the intermediate set. We can calculate the form of the intermediate set also when the game satisfies the converse condition, that is, when $v$ is submodular.

\begin{lemma}\label{Lemma supersub}
 If $v$ is submodular, then $\Mcalv=\{\bx^v(\pi)\mid \pi \in \Pi_n\}$.
\end{lemma}
\begin{proof}
	Let $v$ be a submodular game. By Theorem \ref{Theorem Bv} we need to prove that any nonempty $\calM_{\calH}(v)$ contains a unique allocation, which is necessarily some marginal vector. Let $\calH=(C_1,\dots,C_k)$ be a~chain and $\bx\in\calM_{\calH}(v)$. Take any $i\in\{1,\dots,k\}$ and a~coalition $A\subseteq C_i\setminus C_{i-1}$. We obtain
	$$
	\aligned
	v(C_i) - v(C_{i-1}) &\stackrel{\eqref{Label Bv 1}}{=} \bx(C_i\setminus C_{i-1}) = \bx((C_i\setminus C_{i-1})\setminus A) + \bx(A) \\
	&\stackrel{\eqref{Label Bv 2}}{\ge} v(C_{i-1}\cup ((C_i\setminus C_{i-1})\setminus A)) + v(C_{i-1}\cup A)-2\cdot v(C_{i-1}) \\
	&\ge v(C_i) - v(C_{i-1}),
	\endaligned
	$$
	where the last inequality follows from  submodularity of $v$. This yields
\begin{equation}\label{eq:sbm}
	\bx(A)=  v(C_{i-1} \cup A) -v(C_{i-1}).
\end{equation}
	Let $j\in C_i\setminus C_{i-1}$ and let $B\subseteq C_i\setminus C_{i-1}$ be such that $j\notin B$. It follows from the identity $x_j=\bx(B\cup\{j\}) - \bx(B)$ and from \eqref{eq:sbm} that
	$$
	x_j  =  v(C_{i-1}\cup B\cup \{j\}) -v(C_{i-1}\cup B).
	$$
	Since the above equality holds true for all $i$, $j$ and $B$ as specified above, this immediately implies that $\bx$ is a marginal vector. \qed
\end{proof}

\begin{remark}[Permission structures]
	Any coalitional chain $\calH$ enriches the player set $N$ with an additional structure. Specifically we will briefly mention that this is a special case of a permission structure; see \cite{Gilles10,Gilles92} for details. A \emph{permission structure} on $N$ is a mapping $S\colon N\to 2^N$ satisfying the following condition: $j\in S(i)$ implies that $i\notin S(j)$, for each $i,j\in N$. The players in $S(i)$ are said to be the \emph{successors} of $i\in N$. A \emph{game with permission structure} is a triple $(N,v,S)$, where $N$ is a player set, $v$ is a game and $S$ is a permission structure. A~permission structure is \emph{acyclic} (or \emph{strict}) if there is no sequence of players $i_1,\dots,i_m\in N$ such that $i_1=i_m$ and $i_{k+1}\in S(i_k)$ for every $k=1,\dots,m-1$. Every coalitional chain $\calH=\{C_1,\dots,C_k\}$ gives rise to an acyclic permission structure on $N$. Indeed, it suffices to define a mapping $S_\calH\colon N\to 2^N$ by
	\[
	S_\calH (i) = \begin{cases}
	C_{\ell(i)+1}\setminus C_{\ell(i)} & \ell(i)<k,\\
	\emptyset & \ell(i)=k.
	\end{cases}
	\]
	where $\ell(i)$ is the smallest integer $j$ such that $i\in C_{j}$. Then $S_\calH$ is an acyclic permission structure.
	
	% The collection of all autonomous coalitions\footnote{A coalition $A\subseteq N$ is \emph{autonomous} in a permission structure $S$ if $A\cap S(N\setminus A)=\emptyset$. The autonomous coalitions can act without permission of exterior players.} in $S_\calH$ %coincides with $2^{C_1} \cup \{N\}$. 
%can be expressed as union of all sets $C_i\cup B$ for some $i$ and some $B\subseteq C_{i+1}\setminus C_i$.
	
	Whereas a permission structure $S$ in $(N,v,S)$ is usually determined by an~a~priori known hierarchy among players, we make no such assumption in this paper. By contrast the computation of intermediate set according to \eqref{eqThm1} is based on all the hierarchies among players that are expressible by coalitional chains $\calH$. Thus, there is no preferred coalitional chain $\calH$, albeit only the ones with $\calM_{\calH}(v)\neq \emptyset$ matter. The latter condition makes it possible to claim that, in a sense, any game generates a family of permission structures $S_{\calH}$ satisfying $\calM_{\calH}(v)\neq \emptyset$.
	\exampleEnd
\end{remark}

\subsection{Characterization by marginal games}
We will now show that the intermediate set of any game $v$ can be realized as a finite union of cores of certain games associated with $v$ and chains $\calH$.
\begin{definition}\label{def marg game}
	Let $v\in\Gamma(N)$ and let $\calH=\{C_1,\dots,C_k\}$ be a coalitional chain. For each $i=1,\dots,k$ we define a game $v_i^\calH$ with the player set $C_i\setminus C_{i-1}$:
	$$
	v_i^\calH(B):=v(C_{i-1}\cup B)-v(C_{i-1}), \qquad \text{for all $B\subseteq C_i\setminus C_{i-1}$.}
	$$
	An \emph{$\calH$-marginal game} is the game $v^\calH$ with the player set $N$, where $$v^\calH(B) := \sum_{i=1}^k v_i^\calH(B\cap (C_i\setminus C_{i-1})), \qquad \text{for each $B\subseteq N$.}$$
\end{definition}
Thus, given a game $v$ and a chain $\calH$, the $\calH$-marginal game $v^\calH$ measures aggregated marginal coalitional contributions to all the blocks $C_i\setminus C_{i-1}$. 

\begin{remark}
	In order to define the solution concept called equal split-off set, 
	Branzei et al. introduced in \cite[Section 4.2.1]{Tijs05} the following notion of marginal game based on a game $v$ and an ordered partition. We will reformulate it equivalently using a chain $\calH=\{C_1,\dots,C_k\}$. Define a game $\bar{v}_i^\calH$ with the player set $N\setminus C_{i-1}$ by  $\bar{v}_i^\calH(B):=v(C_{i-1} \cup B)-v(C_{i-1})$, for all $i=1,\dots,k$ and each $B\subseteq N\setminus C_{i-1}$. It is obvious that this definition of marginal game is different from $v_i^\calH$ given above.\exampleEnd
\end{remark}

\noindent
The following lemma says that every  component $\calM_{\calH}(v)$ of the intermediate set is the core of an~$\calH$-marginal game.

\begin{lemma}\label{lemMarg}
 $\calM_\calH(v)=\Ccal(v^\calH)$, for any coalitional chain $\calH=\{C_1,\dots,C_k\}$.
\end{lemma}
\begin{proof}
	According to the definition of $v^\calH$, we have
	\begin{equation}\label{Label vt}
	v^\calH(B) = \sum_{i=1}^k v_i^\calH(B\cap (C_i\setminus C_{i-1})) = \sum_{i=1}^k\left[v(C_{i-1}\cup (B\cap (C_{i}\setminus C_{i-1}))) - v(C_{i-1})\right]
	\end{equation}
	for any $B\subseteq N$. Let $\bx\in\Ccal(v^\calH)$ be arbitrary and consider any $i=1,\dots,k$ and any $B\subseteq C_i\setminus C_{i-1}$. Then from the definition of core and \eqref{Label vt} we obtain
	\begin{equation}\label{Label vt 1}
	\bx(B)\ge v^\calH(B) = v(C_{i-1}\cup B) - v(C_{i-1}).
	\end{equation}
	By taking $B=C_i\setminus C_{i-1}$,
	\begin{equation}\label{Label vt 2}
	\bx(C_i\setminus C_{i-1})\ge v(C_i) - v(C_{i-1}).
	\end{equation}
	Plugging $B=N$ into \eqref{Label vt} and once more using the definition of core yield
	$$
	\bx(N) = v^\calH(N) = \sum_{i=1}^k\left[v(C_i) - v(C_{i-1})\right],
	$$
	which implies together with \eqref{Label vt 2},
	$$
	\bx(C_i\setminus C_{i-1})= v(C_i) - v(C_{i-1}).
	$$
	But the last equality and \eqref{Label vt 1} means that $\bx\in\calM_\calH(v)$.
	
	Conversely, assume $\bx\in\calM_\calH(v)$ and fix any $B\subseteq N$. Then
	\begin{equation*}
    \begin{split}
	\bx(B) = &\sum_{i=1}^k \bx(B\cap (C_i\setminus C_{i-1}))\ge \\ &\sum_{i=1}^k\left[v(C_{i-1}\cup (B\cap (C_i\setminus C_{i-1}))) - v(C_{i-1})\right] =v^\calH(B),
		\end{split}
	\end{equation*}
	where the inequality follows directly from \eqref{Label Bv 2} and the last equality from \eqref{Label vt}. If $B=N$ then the inequality above becomes an equality by \eqref{Label Bv 1}, which means that $\bx\in\Ccal(v^\calH)$ holds true. \qed
\end{proof}

\begin{theorem}\label{Thm IS Cores}
	Let $v\in\Gamma(N)$. Then
	\begin{equation*}%\label{eqThm1}
	\Mcalv=\bigcup_{\calH\in\scrC}\Ccal(v^\calH).
	\end{equation*}
\end{theorem}
\begin{proof}
	It suffices to combine Theorem \ref{Theorem Bv} with Lemma \ref{lemMarg}.
\end{proof}

\section{Properties of Intermediate Set}\label{Section properties}

In this section the intermediate set is compared with the core and the Weber set. We list important properties and show whether they are satisfied for those solution concepts. Further, we briefly discuss the relation of the intermediate set to other set-valued solutions.

\subsection{Comparison with the core and the Weber set}
The properties of the intermediate set are summarized in Table \ref{Table properties}. We follow the approach presented in \cite[Section 8.11]{PelegSudholter07}, where numerous properties and solution concepts are listed together with  conditions  under which a certain property is satisfied by a given solution concept. For the reader's convenience we repeat the definitions and include the known properties of the core and the Weber set.

\begin{definition}\label{Def games properties}
	Let $\emptyset\neq \Omega\subseteq \Gamma(N)$. We say that a solution $\sigma\colon \Omega\to 2^{\R^n}$ satisfies
\begin{itemize}
\item \emph{nonemptiness (NE)} if $\sigma(v)\neq \emptyset$ for every $v\in\Omega$;
\item \emph{convex--valuedness (CON)} if $\sigma(v)$ is convex for every $v\in\Omega$;
\item \emph{Pareto optimality (PO)} if $\bx(N)=v(N)$ for every $v\in\Omega$ and every $\bx\in\sigma(v)$;
\item \emph{individual rationality (IR)} if $x_i\ge v(\{i\})$ for every $i\in N$, every $v\in\Omega$ and  every $\bx\in\sigma(v)$;
\item \emph{superadditivity (SUPA)} if $\sigma(v_1)+\sigma(v_2)\subseteq \sigma(v_1+v_2)$ for every $v_1,v_2\in\Omega$ such that $v_1+v_2\in\Omega$;
\item \emph{subadditivity (SUBA)} if $\sigma(v_1)+\sigma(v_2)\supseteq \sigma(v_1+v_2)$ for every $v_1,v_2\in\Omega$  such that $v_1+v_2\in\Omega$;
\item \emph{additivity (ADD)} if $\sigma$ is both subadditive and superadditive;
\item \emph{anonymity (AN)} if $\pi(\sigma (v))=\sigma (\pi v)$ for every $v\in\Omega$ and $\pi\in\Pi_n$ such that $\pi v\in\Omega$, where $\pi(\sigma (v))=\{(x_{\pi(1)},\dotsc,x_{\pi(n)})\mid \bx\in\sigma(v)\}$ and $\pi v$ is defined by $\pi v(A)=v(\pi^{-1}(A))$, $A\subseteq N$;
\item \emph{equal treatment property (ETP)} if $x_i=x_j$ for every $\bx\in\sigma(v)$, every $v\in\Omega$ and any \emph{substitutes}  $i,j\in N$ in $v$, that is, $v(A\cup\{i\})=v(A\cup \{j\})$, for each $A\subseteq N\setminus \{i,j\}$;
\item \emph{reasonableness (RE)} if for every $v\in\Omega$ and for every $\bx\in\sigma(v)$ we have $b_i^{min}(v)\le x_i\le b_i^{max}(v)$ for all $i\in N$, where
$$
\aligned
b_i^{min} &= \min_{A\subseteq N\setminus\{i\}} (v(A\cup \{i\}) - v(A)),\\
b_i^{max} &= \max_{A\subseteq N\setminus\{i\}} (v(A\cup \{i\}) - v(A));
\endaligned
$$
\item \emph{covariant under strategic equivalence (COV)} if for every $v,w\in\Omega$, every $\alpha>0$ and every additive game $z$ such that $w=\alpha v+z$, we have $\sigma(w) = \alpha\sigma(v) + \{(z(\{1\}),\dotsc,z(\{n\})\}$;
\item \emph{null player property (NP)} if for every $v\in\Omega$ and every $\bx\in \sigma(v)$, we have $x_i=0$ whenever player $i$ is a \emph{null player}, that is, $v(A\cup \{i\})=v(A)$ for all $A\subseteq N$;
\item \emph{dummy property (DUM)} if for every $v\in\Omega$ and every $\bx\in \sigma(v)$ we have $x_i=v(\{i\})$ whenever player $i$ is a \emph{dummy player}, that is, $v(A\cup \{i\})=v(A)+v(\{i\})$ for all $A\subseteq N\setminus\{i\}$.
\end{itemize}
\end{definition}

\begin{table}[!ht]
\centering
\begin{tabular}{l|lll}
& $\Ccalv$ & $\Mcalv$ & $\Wcalv$ \\\hline
Nonemptiness & $\bullet$  & $\checkmark$ & $\checkmark$ \\ 
Convex--valuedness & $\checkmark$ &  & $\checkmark$ \\
Pareto optimality & $\checkmark$ & $\checkmark$ & $\checkmark$\\
Individual rationality & $\checkmark$ & $\bullet$ & $\bullet$ \\\hline
Superadditivity  & $\checkmark$ &  &  \\
Subadditivity  &  & $\checkmark$ & $\checkmark$ \\
Additivity  &  &  & \\\hline
Anonymity  & $\checkmark$ & $\checkmark$ & $\checkmark$ \\  
Equal treatment property &  &  &  \\
Reasonableness  & $\checkmark$ & $\checkmark$ & $\checkmark$ \\
Covariance  & $\checkmark$ & $\checkmark$ & $\checkmark$ \\ 
Null player property  & $\checkmark$ & $\checkmark$ & $\checkmark$ \\
Dummy property  & $\checkmark$ & $\checkmark$ & $\checkmark$
\end{tabular}
\caption{Fulfillment of selected properties. The mark $\checkmark$ means that the property is satisfied on $\Omega=\Gamma(N)$, while $\bullet$ means that  only a ``significant'' subclass of games $\Omega\subsetneq \Gamma(N)$ has the corresponding property. The empty space indicates that the property is not satisfied by every game. }
\label{Table properties}
\end{table}
\begin{comment}
\\\hline
Reduced game property (RGP) & Y & N & \\
Weak RGP (WRGP) & Y & N & \\  
Reconfirmation property (RCP) & Y &   & \\  
Converse RGP (CRGP) & Y & & \\  
Aggregate monotonicity (AM) & Y & Y & Y \\  
Strong aggregate monotonicity (SAM) & Y & Y & Y
\end{comment}

Not all the proofs are presented here. We included only those which are nontrivial, important or use the  concepts of nonsmooth calculus. In all other cases the reader is referred to an analogous comparison \cite[Table 8.11.1]{PelegSudholter07}. Well-known facts about the core are included in Table \ref{Table properties} for the sake of completeness.

\begin{lemma}
Both $\calM$ and $\calW$ satisfy NE. 
\end{lemma}
\begin{proof}
Since the limiting and the Clarke superdifferential of a Lipschitz function are nonempty by \cite[Corollary 8.10, Theorem 9.13]{rockafellar.wets.1998}, both $\Mcalv$ and $\Wcalv$ are nonempty for any game $v$. \qed
\end{proof}

It follows directly from the corresponding definitions that both $\calC$ and $\calW$ satisfy CON. By contrast, the set $\Mcalv$ need not be convex; see Example~\ref{Example gloves simple}. Since PO is satisfied by $\calW$, it is also satisfied by any smaller solution concept.

The next lemma states that property IR of both $\calM$ and $\Wcal$ characterizes the so-called \emph{weakly superadditive (zero-monotonic) games}, which are defined as elements of 
$$
\Gamma^*(N)=\{v\in\Gamma(N)\mid v(A\cup \{i\})\ge v(A)+ v(\{i\})\text{ for all $A\subseteq N$  and $i\in N\setminus A$}\}.
$$

\begin{lemma}\label{Lemma Gamma star}
Let $\emptyset\neq \Omega \subseteq \Gamma(N)$. Then the following three claims are equivalent:
\begin{enumerate}
\item $\Omega\subseteq\Gamma^*(N)$;
\item $\calM$ satisfies IR on $\Omega$;
\item $\Wcal$ satisfies IR on $\Omega$.
\end{enumerate}
\end{lemma}
\begin{proof}
Let $v\in \Omega\subseteq \Gamma^*(N)$. Then for any marginal vector $\bx$ and every $i\in N$, there exists $A\subseteq N\setminus\{i\}$ such that
\[
x_i = v(A\cup\{i\}) -v(A) \ge v(\{i\}).
\]
Since $\Wcalv$ is the convex hull of all the marginal vectors, we have shown that~$\Wcal$ satisfies IR on $\Omega$. This implies that $\calM$ satisfies IR on $\Omega$ as well.

Conversely, assume that $\calM$ satisfies IR on some family of games $\Omega\subseteq\Gamma(N)$. By way of contradiction, let $\Omega\not\subseteq \Gamma^*(N)$. Then there exists $v\in \Omega$, some $A\subseteq N$ and $i\in N\setminus A$ such that $ v(A\cup\{i\}) -v(A) < v(\{i\})$. But this means that there is a marginal vector $\bx$ satisfying
$$
x_i = v(A\cup\{i\}) -v(A) < v(\{i\}).
$$
Since every marginal vector lies in $\Mcalv$, we have arrived at a contradiction with IR of $\calM$ on $\Omega$, and the proof is finished. 
\qed\end{proof}

Concerning SUPA, SUBA and ADD, the proofs are consequences of the general results about superdifferentials/subdifferentials.

\begin{lemma}
 $\calM$ and $\calW$ are subadditive and none of them is additive, in general.
\end{lemma}
\begin{proof}
It suffices to apply Proposition \ref{Proposition sum rule} from Appendix \ref{App:Super} to the Lov\' asz extension of a game.\qed
\end{proof}

Anonymity holds true for both $\calM$ and $\calW$ by Proposition \ref{Proposition Cv Wv superdifferential} and Lemma~\ref{Lemma sets zero one}, since all the discussed superdifferentials have an analogous property. Since ETP is in general violated by $\calC$, it cannot hold for any larger solution concept. Similarly, property RE is true for $\calW$ and thus for any solution $\sigma$ included in~$\calW$.

\begin{lemma}
 $\calC$, $\calM$ and $\calW$ satisfy COV.
\end{lemma}
\begin{proof}
Let $v,w\in\Omega$ and $\alpha>0$ be such that $w=\alpha v+z$, where $z$ is an additive game. 
Since the mapping $v\in \Gamma(N) \mapsto \hat{v}$ is linear, we obtain
$$
\hat{w}(\bx) = \alpha\Gl(\bx) + \hat{z}(\bx), \quad \bx\in\R^n.
$$
As $z$ is an additive game, its Lov\'asz extension $\hat{z}$ is a linear function. The sought  result is then  a~consequence of Proposition \ref{Proposition sum rule} from Appendix \ref{App:Super}.
\qed\end{proof}

As regards the null player property, it is easy to see that the Weber set has NP. 
Hence, it follows from the inclusion \eqref{Label sets inclusions} that the intermediate set has NP as well. Since NP and COV imply DUM by \cite[Remark 4.1.18]{PelegSudholter07}, we have completed the whole Table \ref{Table properties}.

\subsection{Relation to other solution concepts}
We will briefly comment on the relation between the intermediate set and selected solution concepts for coalitional games. Our sample contains only the solutions which bear a formal resemblance to the intermediate set or those containing the core. We omit the discussion of the solution concepts whose position with respect to the intermediate set is clear due to a known result, such as the selectope, which is always at least as large as the Weber set \cite{DerksHallerPeters00}. For the sake of brevity we do not repeat definitions of the discussed solutions, but refer to the literature instead.

\begin{description}
\item[\textbf{Solutions for Coalition Structures}]{ A \emph{coalition structure} in an $n$-person game is a  partition $\{B_1,\dotsc,B_m\}$ of the player set $N$. Each coalition structure of Aumann and Dreze \cite{AumannDreze74} induces the core solution with respect to that coalition structure. Since any coalitional chain $\calH=\{C_1,\dots,C_k\}$ generates an ordered partition\footnote{Observe that the converse statement is true as well: Any ordered partition with nonempty blocks is associated with a unique coalitional chain.} $(C_1,C_2\setminus C_1,\dots,C_k\setminus C_{k-1})$, one can ask if there is any relation between the core of coalition structures and the intermediate set. The closer look reveals fundamental differences, however. Namely the payoff vectors $\bx$ associated with games on coalition structures usually satisfy Pareto optimality locally, that is, $\bx(B_i)=v(B_i)$ for each block $B_i$ of the partition. This is certainly not the case of a~payoff $\bx\in\calM_\calH(v)$ since \eqref{Label Bv 1} means that the coalition $C_i\setminus C_{i-1}$ allocates the worth $v(C_i)-v(C_{i-1})$ among its members. Another point of dissimilarity is that the condition $\bx(A)\geq v(A)$ with $A\subseteq N$ applies across all the blocks of partition in the core of a game with a coalition structure $\{B_1,\dotsc,B_m\}$, while the condition \eqref{Label Bv 2} is used only for the sub-coalitions $B$ of each block $C_i\setminus C_{i-1}$.}
\item[\textbf{Equal Split-Off Set (ESOS)}]{This solution concept is based on ordered partitions and may attain non-convex values; see \cite[Section 4.2]{Tijs05}. It follows from Example \ref{Example gloves simple} that $\Mcalv$ is not contained in the ESOS of $v$. Moreover, the additive game from Example 4.2(iv) in \cite{Tijs05} shows that ESOS is not a~part of $\calM$ either.}
\item[\textbf{Equal Division Core (EDC)}]{The solution EDC is another non-convex solution concept, which was introduced by Selten in \cite{Selten72} and consists of ``efficient payoff vectors for the grand coalition which cannot be improved upon by the equal division allocation of any subcoalition''. Using Example~\ref{Example gloves simple} we can show that the EDC of $v$ does not contain and is not contained in $\Mcalv$. Indeed, the EDC of this game coincides with the set $$\{\bx\in\calI(v)\mid \text{$x_1\geq \tfrac 12$ or ($x_2\geq \tfrac 12$ and $x_3\geq \tfrac 12$)}\}.$$}
\item[\textbf{Core Cover (CC)}]{This solution was studied by Tijs and Lipperts \cite{TijsLipperts82}. Example~\ref{Example gloves simple} yields that CC of the glove game coincides with the core and thus it is strictly smaller than the corresponding intermediate set. The converse strict inclusion is rendered by \cite[Example 1]{TijsLipperts82}.}
\item[\textbf{Reasonable Set (RS)}]{See \cite{VaretZamir87} for details. Since the intermediate set has the property RE from Definition \ref{Def games properties}, it holds true that $\Mcalv$ is included in $\textrm{RS}(v)$ whenever $v\in\Gamma^*(N)$.}
\item[\textbf{Dominance Core (DC)}]{The solution DC is defined as the set of all undominated imputations. If $v \in\Gamma^*(N)$ and $\mathrm{DC}(v)\neq\emptyset$, then \cite[Theorem 2.13]{Tijs05} yields $\calC(v)=\textrm{DC}(v)$, which means that $\Mcalv$ contains $\mathrm{DC}(v)$.}
\end{description}
\noindent
In summary, the only remarkable relations are rendered by the last two items: for every game $v\in\Gamma^*(N)$, we have
\(
\mathrm{DC}(v)\subseteq \Mcalv\subseteq \textrm{RS}(v).
\)

\section{Examples}\label{sec:ex}

In this section we simplify the formula \eqref{eqThm1} from Theorem~\ref{Theorem Bv} for two families of games.

\subsection{Simple games}\label{Subsection simple}

We will compute the intermediate set for the class of simple games. The result will be compared with the formula for the core of simple games.  A~game $v\in\Gamma(N)$ is \emph{monotone} if $v(A)\leq v(B)$ whenever $A\subseteq B\subseteq N$ and $v$ is called \emph{simple} if it is monotone, $v(A)\in\{0,1\}$ for each $A\subseteq N$, and $v(N)=1$. Every simple game $v$ over the player set $N$ can be identified with the family $\calV$ of \emph{winning coalitions}  in $v$ as follows:
\[
\calV=\{A\subseteq N\mid v(A)=1\}.
\]
Conversely, any system of coalitions $\calV$ such that $N\in \calV$, $\emptyset\notin \calV$ and
$$
A\subseteq B\subseteq N,\ A\in \calV\Rightarrow B\in \calV,
$$
gives rise to a simple game $v$ by putting $v(A)=1$ if $A\in\calV$ and $v(A)=0$, otherwise. The family of \emph{minimal winning coalitions} in $v$ is 
\[
\calV^m = \{E\in \calV\mid  B\subsetneq E\Rightarrow B\notin \calV, \text{ for every $B\subseteq N$}\}.
\]

Based on the concept of minimal winning coalitions, we will prove that $\Mcalv$ is a union of faces of the standard $(n-1)$-dimensional simplex, where each face corresponds to one minimal winning coalition. For any nonempty coalition $E\subseteq N$, we put
\[
\Delta_E:=\left\{\bx\in\R^n\left|\ \begin{array}{ll}x_i = 0 & \text{if }i\in N\setminus E\\x_i \ge 0 & \text{if }i\in E\\ \bx(E) = 1\end{array}\right.\right\}.
\]

\begin{theorem}\label{Theorem simple Bv}
If $v\in\Gamma(N)$ is a simple game, then
\begin{subequations}\label{Label simple Xv}
\begin{align}
\label{Label simple Cv} \Ccalv &= \bigcap_{E\in \calV^m}\Delta_E, \\
\label{Label simple Bv} \Mcalv &= \bigcup_{E\in \calV^m}\Delta_E.
\end{align}
\end{subequations}
\end{theorem}
\begin{proof}
The formula for core on simple games  \eqref{Label simple Cv} can be derived easily; see \cite[Example~X.4.6]{Owen95}, for instance. If $\bx\in \bigcup_{E\in \calV^m}\Delta_E$, then there is some $E\in \calV^m$ such that $\bx\in\Delta_E$. Since $E$ is a minimal winning coalition, it is straightforward to show that the chain $\{E,N\}$ and $\bx$ satisfy relations \eqref{Label Bv 1}--\eqref{Label Bv 2}, which implies $\bx\in \Mcalv$.

Conversely, assume that $\bx\in \Mcalv$. By Theorem \ref{Theorem Bv} there is a chain $\calH=\{C_1,\dots,C_k\}$ such that $\bx\in\calM_{\calH}(v)$. Let $l$ be the smallest integer satisfying $v(C_l)=1$. Consider now any $E\in\calV^m$ with $E\subseteq C_l$. From \eqref{Label Bv 2} with $i=l$ we see that $\bx(E\cap (C_l\setminus C_{l-1}))\ge 1$, which gives $\bx(E\cap (C_l\setminus C_{l-1}))=1$ by \eqref{Label Bv 1}. Since $\bx(N)=1$, we have 
$\bx(N\setminus (E\cap (C_l\setminus C_{l-1})))=0$. But since $x_i\ge 0$ for all $i\in N$, this implies that $\bx(N\setminus E)=0$, which finishes the proof.
\qed\end{proof}

Assume now that a simple game $v$ is \emph{zero-normalized}, that is, $v(\{i\})=0$ for every player $i\in N$. Then the formula \eqref{Label simple Bv} further simplifies as 
\begin{equation*}%\label{Label simple Mv case}
\Mcalv = \bigcup_{E\in \calV^m}\left\{\bx\in\calI(v) \mid  x_i = 0  \text{ for every $i\in N\setminus E$} \right\},
\end{equation*}
where $\calI(v)$ is the set of all imputations \eqref{imputations}.
Thus, we have $\Mcalv\subseteq \calI(v)$ in this case. Note that the last inclusion also follows from Lemma \ref{Lemma Gamma star} by zero-monotonicity of $v$. The following example shows that $ \calI(v)\subsetneq \Mcalv$ for a particular simple game that is not zero-monotonic.

\begin{example}
Let $n=3$ and 
$$
v(A) = \begin{cases} 1 &\text{if $A\supseteq \{1\}$ or $A\supseteq\{2,3\}$}, \\ 0 &\text{otherwise,}\end{cases} 
\qquad \text{for each $A\subseteq \{1,2,3\}$.}
$$
It is easy to see that
\begin{align*}
\Ccalv &= \emptyset, \\
\calI(v) &=\{ (1,0,0)\}, \\
\Mcalv &= \{(1,0,0)\}\cup\conv\{(0,1,0), (0,0,1)\}.
\end{align*}
Note that this example shows that $\Mcalv$ need not be a connected set in $\dR^n$.
\exampleEnd\end{example}

\begin{remark}\label{rem:question}
Formulas \eqref{Label simple Cv}--\eqref{Label simple Bv} are also interesting from the perspective of variational analysis. On the one hand, the limiting superdifferential is a~union of the Fr\'echet superdifferentials with respect to a suitable neighborhood (Definition \ref{Definition superdifferentials}). On the other hand, the previous theorem states that in a special case the Fr\'echet superdifferential can be written as an intersection of the limiting superdifferentials. This is a relation which does not hold true in general.
\exampleEnd\end{remark}

\subsection{Glove game}\label{Subsection glove}

In the previous subsection we have managed to compute $\Mcalv$ for the class of simple games. In this subsection we will perform the same task for the glove game, which belongs to the class of assignment games \cite{ShapleyShubik71}. In the glove game, there are $n=p+q$ players and each of them has a glove: either a left one or a~right one. When a subset of players forms a coalition, then their joint profit is the number of glove pairs owned together. Specifically, assume that $L$ is the set of all players having the left glove and $R$ is the set of all players having the right glove. Then
$$
v(A) = \min\{\nrm{A\cap L}, \nrm{A\cap R}\}.
$$
We always assume that $L=\{1,\dots,p\}$, $R=\{p+1,\dots,p+q\}$ and  $p\ge q$, without loss of generality.

The shape of core of glove game is known since it is just a special case of an assignment game; see \cite[Section 3.3]{ShapleyShubik71}. Nevertheless, in order to compare the two solution concepts, we will state the known formula for $\Ccalv$. 

\begin{proposition}\label{Proposition C p >= q}
If $p>q$, then $\Ccalv$ consists of a single point $\bx$ whose coordinates are: $x_l=0$ for all $l\in L$ and $x_r=1$ for all $r\in R$. If $p=q$, then $\Ccalv=\conv \{\chi_L,\chi_R\}$.
\end{proposition}

We will provide a simple way of determining the solution of \eqref{Label Bv 1}--\eqref{Label Bv 2}.

\begin{lemma}\label{Lemma B p q any}
Let  $(C_1,\dotsc,C_k)$ be a chain. Given $i=1,\dots,k$, let $p_i$ and $q_i$ be the number of left and right gloves, respectively, owned by $C_i$. Set $p_0=q_0:=0$.
\begin{itemize}
\item If $p_{i-1}=q_{i-1}$ and $p_i=q_i$, then the solution set of system \eqref{Label Bv 1}--\eqref{Label Bv 2} is $$\conv \{\chi_{L\cap (C_i\setminus C_{i-1})},\chi_{R\cap (C_i\setminus C_{i-1})}\}.$$
\item If $p_{i-1}>q_{i-1}$ and $p_i<q_i$, then system \eqref{Label Bv 1}--\eqref{Label Bv 2} does not have a feasible solution.
\item If $p_{i-1}\ge q_{i-1}$ and $p_i\ge q_i$ and at least one inequality is strict, then $\bx$ is a solution to system \eqref{Label Bv 1}--\eqref{Label Bv 2} if and only if $x_l=0$ for all $l\in (C_i\setminus C_{i-1})\cap L$ and $x_r=1$ for all $r\in (C_i\setminus C_{i-1})\cap R$.
\item If $p_{i-1}<q_{i-1}$ and $p_i>q_i$, then system \eqref{Label Bv 1}--\eqref{Label Bv 2} does not have a feasible solution.
\item If $p_{i-1}\le q_{i-1}$ and $p_i\le q_i$ and at least one inequality is strict, then $\bx$ is a solution to system \eqref{Label Bv 1}--\eqref{Label Bv 2} if and only if $x_l=1$ for all $l\in (C_i\setminus C_{i-1})\cap L$ and $x_r=0$ for all $r\in (C_i\setminus C_{i-1})\cap R$.
\end{itemize}
\end{lemma}
\begin{proof}
For this proof, it is more advantageous to work with ordered partitions than with chain, thus we define $B_i:=C_{i}\setminus C_{i-1}$. For the first statement, we realize that $\chi_{L\cap B_i}$ and $\chi_{R\cap B_i}$ solve system \eqref{Label Bv 1}--\eqref{Label Bv 2}. Conversely, denote $\bx$ to be any solution of this system. Then we obtain $\bx(B_i)=p_i-p_{i-1}$ and $\bx(\{l,r\})\ge 1$ for all $l\in L\cap B_i$ and $r\in R\cap B_i$. But summing all these terms results in $\bx(\{l,r\})=1$, which further implies $x_{l_1}=x_{l_2}$ for all $l_1,l_2\in L\cap B_i$ and $x_{r_1}=x_{r_2}$ for all $r_1,r_2\in R\cap B_i$. But this implies the first statement.

Concerning the remaining four statements, we will proof only two of them since the proof of the last two assertions is completely analogous. Assume first that $p_{i-1}=q_{i-1}$ and $p_i>q_i$. Then obviously any $\bx$ with $x_l=0$ for all $l\in B_i\cap L$ and $x_r=1$ for all $r\in B_i\cap R$ satisfies system \eqref{Label Bv 1}--\eqref{Label Bv 2}. Consider now any solution of this system. From \eqref{Label Bv 1} we see that $\bx(B_i)=q_i-q_{i-1}$ and from \eqref{Label Bv 2} we have $\bx(B_i\setminus\{l\})\ge q_i-q_{i-1}$, which together with the nonnegativity of $x_l$ implies $x_l=0$ for all $l\in B_i\cap L$. But this implies one part of the third statement.

Assume that $p_{i-1}>q_{i-1}$ and consider any solution $\bx$ of \eqref{Label Bv 1}--\eqref{Label Bv 2}. Then we have
\begin{equation}\label{Label Bv part 1}
\bx(B_i) = \min\{p_i,q_i\} - \min\{p_{i-1},q_{i-1}\} = \min\{p_i,q_i\} - q_{i-1}.
\end{equation}
Taking $B=\{r\}$ for any $r\in B_i\cap R$ results in $x_r\ge 1$. Similarly, by taking $B=\{l\}$ for $l\in B_i\cap L$ we get $x_l\ge 0$. This results in
\begin{equation}\label{Label Bv part 2}
\bx(B_i) = \bx(B_i\cap L) + \bx(B_i\cap R)\ge 0 + (q_i - q_{i-1}) = q_i - q_{i-1}.
\end{equation}
Combining formulas \eqref{Label Bv part 1} and \eqref{Label Bv part 2} leads to
\begin{equation}\label{Label Bv part 3}
q_i \le \min\{p_i,q_i\}.
\end{equation}

If $p_i<q_i$, then formula \eqref{Label Bv part 3} cannot be satisfied and thus, system \eqref{Label Bv 1}--\eqref{Label Bv 2} does not have any feasible solutions. On the other hand, if $p_i\ge q_i$, then from \eqref{Label Bv part 1} we see that $\bx(B_i)=q_i-q_{i-1}$, and \eqref{Label Bv part 2} further implies that $\bx(B_i\cap L) = 0$ and $\bx(B_i\cap R) = q_i-q_{i-1}$. But this means that $x_r=1$ for all $r\in B_i\cap R$ and one inclusion has been proved.

To finish the proof, we must show that for $p_i\ge q_i$ and for $\bx$ with $x_l=0$ for all $l\in B_i\cap L$ and $x_r=1$ for all $r\in B_i\cap R$, the payoff vector $\bx$ solves \eqref{Label Bv 1}--\eqref{Label Bv 2}. Then 
$$
\aligned
\bx(B_i) &= q_i-q_{i-1} = \min\{p_i,q_i\} - \min\{p_{i-1},q_{i-1}\} = v(C_i)-v(C_{i-1}).
\endaligned
$$
Consider any $B\subseteq C_i\setminus C_{i-1}=B_i$ and assume that $B$ contains $a$ players with left gloves and $b$ players with right gloves. Then
$$
\aligned
\bx(B)&=b\ge\min \{a+p_{i-1}-q_{i-1},b\} = \min\{a+p_{i-1},b+q_{i-1}\} - q_{i-1}\\
&=\min\{a+p_{i-1},b+q_{i-1}\} - \min\{p_{i-1},q_{i-1}\} \\
&= v(C_{i-1}\cup B)-v(C_{i-1}),
\endaligned
$$
which concludes the proof.\qed
\end{proof}

We will prove the main theorem of this section. It says that every $\bx\in \Mcalv$ can be generated  via Theorem \ref{Theorem Bv} by choosing coalitions $B_1,\dots,B_{q+1}$ such that: (i) $B_1,\dots,B_{q}$ are $2$-player coalitions containing a pair of players each of which owns one right and one left glove, respectively, (ii) the coalition $B_{q+1}$ contains only the players possessing left gloves or $B_{q+1}=\emptyset$ if $p=q$, that is, $B_{q+1}\subseteq L$.

\begin{theorem}\label{Theorem gloves}
	Assume thast $p\ge q$. Then $\bx\in \Mcalv$ if and only if there exists $\tilde{L}\subseteq L$ with $\abs{\tilde{L}}=q$ and a~bijection $\rho\colon\tilde{L}\to R$ such that the following conditions are satisfied:
\begin{subequations}\label{Label Bv gloves}
\begin{align}
\label{Label Bv gloves 1}x_l + x_{\rho(l)} &= 1\text{ for all }l\in\tilde{L},\\
\label{Label Bv gloves 2}1\ge x_l&\ge 0\text{ for all }l\in \tilde{L},\\
\label{Label Bv gloves 3}x_l&=0\text{ for all }l\in L\setminus \tilde{L}.
\end{align}
\end{subequations}
\end{theorem}
\begin{proof}
Let $\bx$ satisfy \eqref{Label Bv gloves 1}--\eqref{Label Bv gloves 3}. We can enumerate the elements of $\tilde{L}$ as $l_1,\dots,l_q$ and define the coalitions
$$
B_1 = \{l_1,\rho(l_1)\}, \dots, B_q = \{l_q,\rho(l_q)\}, B_{q+1} = L\setminus \tilde{L}
$$
and chain $\calH=(C_1,\dots,C_{q+1})$ with $C_i=B_1\cup\dots \cup B_i$ for $i=1,\dots,q+1$. Then it is easy to verify that $\bx\in \Mcalv$ by Theorem \ref{Theorem Bv} and using the chain above.

For the proof of the second inclusion, denote by $p_i$ the number of left gloves owned by players $C_i$ and by $q_i$ the number of right gloves owned by the same players. Put $p_0=q_0:=0$. To prove the statement, we will construct~$\rho$ by a variant of finite induction. There are three possibilities: $p_1=q_1$, $p_1<q_1$ or $p_1>q_1$.

If $p_1=q_1$, then define two sets $L_1:=C_1\cap L$ and $R_1:=C_1\cap R$. The first part of Lemma \ref{Lemma B p q any} states that $\bx$ is a solution to system \eqref{Label Bv 1}--\eqref{Label Bv 2} if and only if there exists $\la\in[0,1]$ such that $x_l=\la$ for all $l\in L_1$ and $x_r=1-\la$ for all $r\in R_1$. Since $\abs{L_1}=\abs{R_1}$, we can define a bijection $\rho\colon  L_1\to R_1$. Now observe that
$$
x_l + x_{\rho(l)} = \la + (1-\la) = 1
$$
for every $l\in L_1$. Hence, \eqref{Label Bv gloves 1}--\eqref{Label Bv gloves 2} holds true for $L_1$.

If $p_1>q_1$, then we deduce from the second part of Lemma \ref{Lemma B p q any} that there are two possibilities: either there exists $i>0$ such that $p_1>q_1,\dots, p_{i-1}>q_{i-1}$ with $p_i=q_i$ or $p_1>q_1,\dots, p_k>q_k$. We will consider only the first possibility and return to the second one at the end of the proof. Let $L_1:=C_i\cap L$ and $R_1:=C_i\cap R$. Due to the third part of Lemma \ref{Lemma B p q any} this implies that $x_l=0$ for all $l\in L_1$ and $x_r=1$ for all $r\in R_1$. But since $p_i-p_0=q_i-q_0$, there is a bijection $\rho$ between $L_1$ and $R_1$ and, similarly as in the case $p_1=q_1$, we observe that $x_l+x_{\rho(l)}=1$ and $x_l\ge 0$ for all $l\in L_1$.

The case of $p_1<q_1$ can be handled in exactly the same way. Note that it may not happen that $p_1<q_1,\dots, p_k<q_k$ because there are more left gloves than right gloves.

Applying this procedure multiple times, we have managed to find an index~$i$, sets $\hat{L}$ and $\hat{R}$ and a bijection $\rho\colon\hat{L}\to\hat{R}$ such that the following properties are satisfied:
\begin{enumerate}
\item $p_i=q_i$ and $p_{i+1}>q_{i+1}, \dots, p_k>q_k$,
\item $x_l+x_{\rho(l)}=1$ and $x_l\ge 0$ for all $l\in\hat{L}$,
\item $\hat{L}\cup\hat{R}=C_i$ and  $\abs{\hat{L}}=\abs{\hat{R}}=p_i$.
\end{enumerate}
From the third part of Lemma \ref{Lemma B p q any} we obtain then that $x_l=0$ for all $l\in L\setminus\hat{L}$ and $x_r=1$ for all $r\in R\setminus\hat{R}$. Find any $L'\subseteq L\setminus\hat{L}$ such that $\abs{L'}=\abs{R\setminus\hat{R}}$, define $\tilde{L}:=\hat{L}\cup L'$ and extend bijection $\rho\colon \hat{L}\to\hat{R}$ to a bijection $\rho\colon\tilde{L}\to R$. Then any such $\tilde{L}$ and $\rho$ satisfy \eqref{Label Bv gloves 1}--\eqref{Label Bv gloves 3}, which completes the proof.\qed
\end{proof}

\section{Conclusions} \label{sec:concl}
We have inserted a new non-convex solution concept, the intermediate set, in-between the core and the Weber set of a coalitional game. Our main tool in this paper is Theorem \ref{Theorem Bv}, which transforms the analytical task of computing the limiting superdifferential into the equivalent problem of solution of finitely many systems of linear inequalities. The achieved characterization by coalitional chains makes it possible to interpret the payoff vectors in the intermediate set as marginal coalitional contributions satisfying the conditions \eqref{Label Bv 1}--\eqref{Label Bv 2} or, equivalently, as core allocations with respect to a family of marginal games (Theorem \ref{Thm IS Cores}).

We will outline some ideas for the future research on this topic:
\begin{enumerate}
	\item The family of all coalitional chains in the player set $N$  is in one-to-one correspondence with the set of all nonempty faces of the permutohedron of order $n$; see \cite{Ziegler95}. An idea is to look at the relation between the algebraic structure of the corresponding face lattice and the geometric composition of the convex components $\calM_\calH(v)$ of the intermediate set $\Mcalv$. Specifically, what are the properties of the mapping sending a chain $\calH$ to the convex polytope  $\calM_\calH(v)$ for a fixed game $v$?
	\item The following question was already mentioned in Remark \ref{rem:question} using the language of variational analysis: For which games $v$ is the core of $v$ an intersection of selected components of the intermediate set? Example \ref{Example gloves simple}, Example \ref{Example star} and Theorem~\ref{Theorem simple Bv} suggest that non-trivial examples of such coalitional games $v$ are  not difficult to find.
	\item Many solution concepts (the core, the Shapley value etc.) can be axiomatized on various classes of games. Is there an axiomatization of the intermediate set on some class of coalitional games?
	\item The coincidence of the core with the Weber set is essential for the characterization of extreme rays of the cone of supermodular games presented in~\cite{StudenyKroupa:CoreExtreme}.  There can be a large gap between the core and the Weber set outside the family of supermodular games. Our plan is to study the geometrical properties of the intermediate set on larger classes of games including the supermodular games, such as the cone of exact games or the cone of superadditive games.
\end{enumerate}

\appendix
\section*{Appendix}
\section{Superdifferentials}\label{App:Super}

In this section we will define the selected concepts of variational (nonsmooth) analysis, mainly various superdifferentials which generalize the superdifferential of concave functions. Since these superdifferentials will be computed only for the Lov\'asz extension, we will confine to defining superdifferentials only for piecewise affine functions. Even though the computation of these objects may be rather a challenging task, see e.g. \cite{adam.cervinka.pistek.2014,henrion.outrata.2008}, the presented framework allows for a significant simplification. For the general approach based on upper semicontinuous functions, we refer the reader to \cite{rockafellar.wets.1998}.

The standard monographs on variational analysis \cite{mordukhovich.2006,rockafellar.1970,rockafellar.wets.1998} follow the approach usual in convex analysis by dealing with subdifferentials instead of superdifferentials. However, most of the results can be easily transformed to the setting of superdifferentials, usually by reversing inequalities only.

\begin{definition}\label{Definition superdifferentials}
Let $f\colon \R^n\to\R$ be a piecewise affine function and $\bxb\in\R^n $. We say that $\bx^*\in\R^n$ is~a
\begin{itemize}
\item \emph{regular (Fr\'echet) supergradient} of $f$ at $\bxb$
%, denoted by $\bx^*\in\sdf f(\bxb)$,
if there exists neighborhood $\Xcal$ of $\bxb$ such that for all $\bx\in\Xcal$ we have
$$
f(\bx) - f(\bxb) \le \ssoucin{\bx^*}{\bx-\bxb};
$$
\item \emph{limiting (Mordukhovich) supergradient} of $f$ at $\bxb$
%, denoted by $\bx^*\in\sdl f(\bxb)$,
if for every neighborhood $\Xcal$ of $\bxb$ there exists $\bx\in\Xcal$ such that $\bx^*$ is a Fr\'echet supergradient of $f$ at $\bx$;
\item \emph{convexified (Clarke) supergradient} of $f$ at $\bxb$
%, denoted by $\bx^*\in\sdc f(\bxb)$,
if 
$$
\bx^*\in\conv\{\by\in\dR^n|\ \forall\,\text{neighborhood }\Xcal\text{ of }\bxb\enskip \exists\,\bx\in\Xcal\cap D\text{ with }\by = \nabla f(\bx)\},
$$
where
$$
D:= \{\bx\in\R^n|\ f\text{ is differentiable at }\bx\}.
$$
\end{itemize}
The collection of all (regular, limiting, convexified) supergradients of $f$ at $\bxb$ is called \emph{(Fr\'echet, limiting, Clarke) superdifferential} and it is denoted by $\sdf f(\bxb)$, $\sdl f(\bxb)$ and $\sdc f(\bxb)$, respectively.
\end{definition}

\begin{remark}
The previous definition can be found e.g. in \cite[Definition 8.3]{rockafellar.wets.1998}. Note that in the original definition term $o(\norm{\bx-\bxb})$ is added. Because we work with piecewise affine functions, this term is superfluous. If $f$ is concave, then all the above superdifferentials  coincide with the standard superdifferential for concave functions.
\exampleEnd\end{remark}

It is possible to show that
$$
\sdf f(\bxb) \subseteq \sdl f(\bxb) \subseteq \sdc f(\bxb), \quad \bxb\in \R^n,
$$
where all the inequalities may be strict. According to \cite[Theorem 8.49]{rockafellar.wets.1998} we have the following relation between the limiting and the Clarke superdifferential for every piecewise affine function~$f$:
$$
\sdc f(\bxb) = \conv \sdl f(\bxb).
$$
\noindent
We will show the  differences among the three discussed superdifferentials.

\begin{example}
Let $f\colon \R\to\R$ be defined by
$$
f(x) = \begin{cases} x &\text{if $x\in(-\infty,0]$,}\\ 0 &\text{if $x\in[0,1]$,}\\ x-1 &\text{if $x\in[1,\infty)$.} \end{cases}
$$
This function is depicted in Figure \ref{Figure f superdifferentials}.
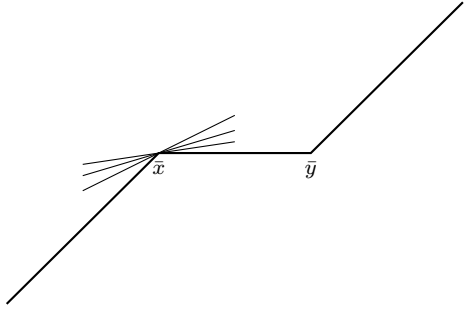
\begin{figure}[!ht]
\centering
\begin{tikzpicture}
\draw [thick] (-2,-2) -- (0,0) node[below] {$\xb$}  -- (2,0) node[below] {$\yb$} -- (4,2);
\draw (-1,-0.5) -- (1,0.5);
\draw (-1,-0.3) -- (1,0.3);
\draw (-1,-0.15) -- (1,0.15);
\end{tikzpicture}
\caption{Supergradients for a piecewise affine function $f$}
\label{Figure f superdifferentials}
\end{figure}
Consider points $\xb=0$ and $\yb=1$. The locally supporting hyperplanes from the definition of Fr\'echet superdifferential at $\xb$ are depicted in the figure. Note that there are no affine majorants for $f$ at $\yb$, which means that the Fr\'echet superdifferential is empty at this point. Thus we obtain
\begin{align*}
&\sdf f(\xb) = [0,1], & \sdf f(\yb) &= \emptyset, \\
&\sdl f(\xb) = [0,1], & \sdl f(\yb) &= \{0,1\}, \\
&\sdc f(\xb) = [0,1], & \sdc f(\yb) &= [0,1]. 
\end{align*} 
\exampleEnd\end{example}

\begin{comment}
Due to Lemma \ref{Lemma sets zero one}, it is sufficient to be able to compute superdifferentials of positively homogeneous functions at $\zero$. For such situation, another simplification is possible, the local character of a superdifferential changes into a global one.

\begin{lemma}
Assume that $f:\R^n\to\R$ is positively homogeneous and piecewise linear function. Then
\begin{align*}
\sdf f(0) &= \{\bx^*|\ f(\bx) \le \ssoucin{\bx^*}{\bx}\text{ for all }\bx\in\R^n\},\\
\sdl f(0) &= \bigcup_{\bx\in\R^n}\sdf f(\bx).
\end{align*}
\end{lemma}
\begin{proof}
The first part is obvious. The second one follows from the fact that $\sdf f(\bx) = \sdf f(c\bx)$ for all $c>0$.\qed
\end{proof}
\end{comment}
The superdifferential sum rule is employed frequently in this paper. The following proposition collects the results of \cite[Exercise 8.8, Corollary 10.9, Exercise 10.10]{rockafellar.wets.1998}.

\begin{proposition}\label{Proposition sum rule}
Let $f_1,f_2\colon\R^n\to\R$ be piecewise affine functions. Then 
$$
\partial(f_1+f_2)(\bx)\subseteq \partial f_1(\bx)+\partial f_2(\bx), \quad \bx \in \dR^n. 
$$
Moreover, if at least one of the functions is smooth around $\bx$, we obtain equality in the previous relation.
\end{proposition}

\section{Proof of Theorem \ref{Theorem Bv}}\label{Appendix proof}

To prove Theorem \ref{Theorem Bv}, consider first a game $v\in\Gamma(N)$, fix $\bxb\in\R^n$ and choose any $\pi\in\Pi(\bxb)$. Then there are necessarily unique integers
$$
0=L_0<L_1<\dots<L_k=n
$$
such that $L_i-L_{i-1}$ is the number of coordinates of $\bxb$ which have the $i$--th greatest distinct value in the order given by $\pi$:
$$
\bxb_{\pi(1)} = \dots = \bxb_{\pi(L_1)}>\bxb_{\pi(L_1+1)} = \dots = \bxb_{\pi(L_2)} > \dots > \bxb_{\pi(L_{k-1}+1)} = \dots = \bxb_{\pi(L_k)}.
$$
Define
$$
C_i:=\{\pi(1),\dots,\pi(L_i)\}
$$
and observe that $C_i$ is independent of the choice of $\pi\in\Pi(\bxb)$. Take any $\bx$ sufficiently close to $\bxb$ and select some $\rho\in\Pi(\bx)$. Then $\rho\in\Pi(\bxb)$ and 
$$
\aligned
V_j^\rho(\bx)&\subseteq V_j^\rho(\bxb),\ &&j=1,\dots,n,\\
V_{L_i}^\rho(\bx)&= V_{L_i}^\rho(\bxb)=C_i,\ &&i=1,\dots,k.
\endaligned
$$
This allows us to write $\Gl$ in a separable structure
\begin{equation}\label{Label Lovasz decomposition}
\Gl(\bx) = \sum_{i=1}^k\Gl_i(\bx_{C_i\setminus C_{i-1}}),
\end{equation}
where $\bx_{A}$ is the restriction of $\bx$ to components $A$ and $\Gl_i:\R^{\abs{B_i}}\to\R$ is defined as
$$
\Gl_i(\by) = \sum_{j=1}^{\abs{B_k}}y_{\phi(j)}\left[v(C_{i-1}\cup V_j^\phi(\by)) - v(C_{i-1}\cup V_{j-1}^\phi(\by))\right],
$$
where $\phi\in\Pi(\by)$.  We now fix a constant $c>0$, coalition $B\subseteq C_i\setminus C_{i-1}$ and denoting $a$ to be the common value of $\bxb$ on $C_i\setminus C_{i-1}$, we obtain
$$
\aligned
\Gl_i((\bxb+c\chi_B)_{C_i\setminus C_{i-1}}) &= a\left[(v(C_i) - v(C_{i-1}\cup B)\right] + (a+c)\left[(v(C_{i-1}\cup B) - v(C_{i-1})\right],\\
\Gl_i(\bxb_{C_i\setminus C_{i-1}}) &= a\left[(v(C_i) - v(C_{i-1})\right],
\endaligned
$$
so that
\begin{subequations}\label{Label app}
\begin{equation}\label{Label app 1}
\Gl_i(\bxb_{C_i\setminus C_{i-1}}+c\chi_B) - \Gl_i(\bxb_{C_i\setminus C_{i-1}}) = c\left[(v(C_{i-1}\cup B) - v(C_{i-1})\right].
\end{equation}
When we choose $B=N$, we can move in the opposite direction as well, obtaining
\begin{equation}\label{Label app 2}
\Gl_i(\bxb_{C_i\setminus C_{i-1}}-c\chi_N) - \Gl_i(\bxb_{C_i\setminus C_{i-1}}) = c\left[(v(C_i) - v(C_{i-1})\right].
\end{equation}
\end{subequations}
Now we prove the following lemma.

\begin{lemma}\label{Lemma Bv}
For any $i\in\{1,\dots,k\}$ we have
\small{
$$
\sdf\Gl_i(\bxb_{C_i\setminus C_{i-1}}) = \left\{\bx^*\left|\ \aligned  \bx^*(C_i\setminus C_{i-1}) &= v(C_i) - v(C_{i-1}),\\ \bx^*(B) &\ge v(C_{i-1}\cup B) - v(C_{i-1})\text{ for all $B\subseteq C_i\setminus C_{i-1}$}\endaligned \right.\right\}.
$$}
\end{lemma}
\begin{proof}
The definition of Fr\'echet superdifferential and the piecewise affinity of $\Gl_i$ give
$$
\sdf\Gl_i(\bxb_{C_i\setminus C_{i-1}}) = \{\bx^*|\ \Gl_i(\by) - \Gl_i(\bxb_{C_i\setminus C_{i-1}})\le \ssoucin{\bx^*}{\by -\bxb_{C_i\setminus C_{i-1}}}\text{ for all }\by\text{ close to }\bxb_{C_i\setminus C_{i-1}}\}.
$$
Consider now any $\bx^*\in\sdf\Gl_i(\bxb_{C_i\setminus C_{i-1}})$,  any $B\subseteq C_i\setminus C_{i-1}$, and put $\by = \bxb_{C_i\setminus C_{i-1}}+c \chi_B$, where $c>0$ is sufficiently small. By realizing that $\ssoucin{\bx^*}{\by -\bxb_{C_i\setminus C_{i-1}}} = c\bx^*(B)$ and from relation \eqref{Label app 1} it follows that
$$
\bx^*(B) \ge v(C_{i-1}\cup B) - v(C_{i-1}).
$$
Similarly from \eqref{Label app 2} we obtain equality in the previous relation for $B=N$. This finishes the proof of the first inclusion.

Consider now any $\bx^*$ from the right--hand side of the formula in Lemma \ref{Lemma Bv} and fix any~$\by$ from a sufficiently small neighborhood of $\bxb_{C\setminus C_{i-1}}$. Defining
$$
\aligned
\by^0 &:= \bxb_{C_i\setminus C_{i-1}}-\chi_{\{1,\dots,\nrm{C_i\setminus C_{i-1}}\}},\\
\by^j &:= \bxb_{C_i\setminus C_{i-1}}+\chi_{\{\phi(1)\dots\phi(j)\}},\quad j=1,\dots,\nrm{C_i\setminus C_{i-1}},
\endaligned
$$
we have
$$
\by\in \conv\left\{\by^0, \by^1, \dots, \by^{\nrm{C_i\setminus C_{i-1}}}\right\}.
$$
From the assumption and from \eqref{Label app} we obtain that 
\begin{equation}\label{Label defin Lovasz restricted}
\Gl_i(\by^j) - \Gl_i(\bxb_{C_i\setminus C_{i-1}})\le \ssoucin{\bx^*}{\by^j -\bxb_{C_i\setminus C_{i-1}}}
\end{equation}
for all $j=0,\dots,\nrm{C_i\setminus C_{i-1}}$. Since $\Gl_i$ is linear on very particular domains and since $\by$ lies in the convex hull of the above points, we obtain that formula \eqref{Label defin Lovasz restricted} holds also for $\by$. This finishes the proof.\qed
\end{proof}
\noindent
The decomposition \eqref{Label Lovasz decomposition} together with Lemma \ref{Lemma Bv} imply  that Theorem \ref{Theorem Bv} holds true.

\end{document}